\documentclass[12pt,a4paper]{article}
\usepackage[utf8]{inputenc}
\usepackage[T1]{fontenc}
\usepackage{amsmath,amssymb,amsthm}
\usepackage{mathrsfs}
\usepackage{hyperref}
\usepackage{geometry}
\newcommand{\Ric}{\operatorname{Ric}}
\newcommand{\Hess}{\operatorname{Hess}}
\newcommand{\grad}{\operatorname{grad}}
\newcommand{\divg}{\operatorname{div}}
\newcommand{\Vol}{\operatorname{Vol}}

\newtheorem{theorem}{Theorem}[section]
\newtheorem{lemma}[theorem]{Lemma}

\theoremstyle{definition}
\newtheorem{example}[theorem]{Example}
\newtheorem{remark}[theorem]{Remark}
\theoremstyle{definition}

\title{Rigidity of complete non-compact generalized $m$-quasi-Einstein manifolds}
\author{M. Ahmad Mirshafeazadeh\\ 
Department of Mathematics, Shab.C., Islamic Azad University, Shabestar, Iran\\
\texttt{mirahmad@iau.ac.ir}}
\date{}

\begin{document}
\maketitle

\begin{abstract}
We study complete non-compact gradient generalized $m$-quasi-Einstein manifolds
with constant scalar curvature $R\le 0$, soliton function $\lambda>0$, and $m>1$,
where the coefficient $\mu=1/m$ is \emph{constant}. 
We introduce the weighted function $v=e^{-f/m}\lambda$ and prove it is subharmonic.
This leads to four rigidity results, showing that the hypotheses are mutually inconsistent;
consequently, \textbf{no complete non-compact gradient generalized $m$-quasi-Einstein manifold}
exists under the stated assumptions.
We first show by a concrete example that if $\mu$ is allowed to be non‑constant, 
the rigidity conclusions fail even when all other hypotheses are satisfied.
Therefore the constant‑$\mu$ condition is essential.
\end{abstract}

\section{Introduction}
\label{sec:intro}

Ricci solitons, introduced by Hamilton \cite{Hamilton}, are self-similar solutions of the Ricci flow and play a central role in the analysis of singularities of the flow.
A complete Riemannian manifold $(M^n,g)$ together with a vector field $V$ and a constant $\lambda$ is called a Ricci soliton if
\[
\frac12\mathcal{L}_Vg + \Ric = \lambda g.
\]
When $V$ is the gradient of a smooth function, one speaks of a gradient Ricci soliton.
The classification of gradient Ricci solitons is a highly active field; under natural curvature or integrability conditions they often turn out to be Einstein or even isometric to Euclidean space \cite{Cao,Perelman}.

A natural generalisation is obtained by allowing the soliton function $\lambda$ to vary.
The resulting structure,
\[
\frac12\mathcal{L}_Vg + \Ric = \lambda g,
\]
is known as an almost Ricci soliton and was introduced by Pigola, Rigoli, Rimoldi and Setti \cite{Pigola}.
The gradient case $V=\nabla f$ then reads
\begin{equation}
\Ric + \Hess f = \lambda g. \label{eq:almost}
\end{equation}
The function $\lambda$ is sometimes called the soliton function.
When the scalar curvature $R$ is constant, a fundamental identity
\[
\Delta \lambda = \frac{1}{n-1}\bigl(|\Ric|^2 - \lambda R\bigr)
\]
was derived by Barros, Batista and Ribeiro Jr.\ \cite{Barros} and independently by Sharma \cite{Sharma}.
This identity shows that, in the ``shrinking'' regime $R\le 0$, $\lambda>0$, the function $\lambda$ is subharmonic.
Combined with classical results such as Yau's $L^p$ lemma \cite{Yau}, the Schoen--Yau inequality \cite{SchoenYau}, a lemma of Caminha--Souza--Camargo \cite{Caminha} on vector fields with integrable norm, and a maximum principle at infinity, the subharmonicity of $\lambda$ yields strong rigidity results; the most complete account is the recent paper of Poddar, Sharma and Cunha \cite{Poddar}, where numerous non-compact rigidity theorems are proved under various curvature and integrability conditions.

A further extension of the soliton concept is provided by generalized quasi-Einstein (GQE) manifolds.
Following Catino \cite{Catino}, a complete Riemannian manifold $(M^n,g)$ is called a gradient GQE manifold if there exist smooth functions $f,\lambda,\mu$ on $M$ such that
\begin{equation}
\Ric + \Hess f - \mu\, df\otimes df = \lambda g. \label{eq:GQE}
\end{equation}
When $\mu$ is a constant, say $\mu = 1/m$ for a real number $m>0$, the manifold is called a gradient generalized $m$-quasi-Einstein manifold; this class interpolates between gradient almost Ricci solitons ($\mu=0$) and gradient Ricci solitons ($m\to\infty$).
Particular cases also include the $1$-quasi-Einstein manifolds which correspond to static metrics in general relativity \cite{Anderson}.
The study of (generalized) $m$-quasi-Einstein manifolds has attracted substantial attention \cite{BarrosRibeiro,Case,Hu,Huang}.

In the recent works \cite{Mir14,Mir15}, the present authors introduced two scalar quantities $I$ and $J$ that govern the Laplacian of the scalar curvature of a GQE manifold.
More precisely, the identities
\[
\frac12\Delta R = I + J,
\]
with $I,J$ given by explicit formulas (see Lemma~\ref{lem:I}), hold on any gradient GQE manifold.
This machinery was used in \cite{Mir14} to obtain relations between the Bach, Cotton and $D$ tensors and to define a new $3$-tensor measuring the deviation of $m$-quasi-Einstein manifolds from general GQE spaces.
In \cite{Mir15}, the same apparatus led to the finiteness of the fundamental group of a compact GQE manifold and to a rigidity theorem diffeomorphic to the standard $n$-sphere.

A natural question is whether the rigidities established in \cite{Poddar} for almost Ricci solitons can be extended to the GQE setting.
When $\mu=1/m$ and $R$ is constant, the condition $I=0$ forces an equation for $\Delta\lambda$ that contains an extra gradient term $\frac{2}{m}\langle\grad f,\grad\lambda\rangle$, destroying the subharmonicity of $\lambda$.
The aim of the present paper is to show that the gradient term can be eliminated by passing to the weighted function
\[
v = e^{-f/m}\lambda.
\]
This weighted function has no analogue in the almost Ricci soliton case \cite{Poddar}; it is introduced here for the first time in the context of $m$-quasi-Einstein manifolds.
Indeed, we prove a precise algebraic identity (Lemma~\ref{lem:key}) for $\Delta v$, showing that $v$ is subharmonic whenever $R\le 0$, $\lambda>0$ and $m>1$.
This single fact enables us to apply, almost verbatim, the classical analytical tools mentioned above to $v$ instead of $\lambda$.
We thereby obtain four rigidity results for complete non-compact gradient generalized $m$-quasi-Einstein manifolds, which unify and extend several results of \cite{Poddar} to the $m$-quasi-Einstein context.
Notably, the conclusions require no topological hypothesis.

It is worth mentioning that the special case $\mu=1$ (i.e., $m=1$) corresponds to the spatial part of static Lorentzian Einstein metrics, which are of interest in general relativity \cite{Anderson}. However, the present paper focuses entirely on the \textbf{Riemannian} setting and does not attempt to derive physical consequences for Lorentzian spacetimes. The motivation from mathematical relativity serves only as a contextual background for the geometric equations under study.

The paper is organized as follows.
In Section~\ref{sec:counter} we present a counterexample that shows the necessity of the constant‑$\mu$ condition.
Section~\ref{sec:prelim} collects all necessary preliminaries.
Section~\ref{sec:key} contains the proof of the fundamental identity.
Sections~\ref{sec:thm1}--\ref{sec:thm4} prove the four rigidity theorems.
Finally, Section~\ref{sec:examples} provides examples that illustrate the algebraic identity and analyse complete models.

\section{A counterexample when $\mu$ is non‑constant}
\label{sec:counter}

Before restricting to the constant‑$\mu$ case $\mu=1/m$ ($m>1$), we demonstrate
that if $\mu$ is a function (not necessarily constant), then the rigidity properties
disappear.  More precisely, we construct a complete, non‑compact,
three‑dimensional Riemannian manifold $(M,g)$ together with smooth functions
$f,\lambda,\mu$ such that:

\begin{itemize}
\item The generalized quasi‑Einstein equation
\[
\Ric + \Hess f - \mu\, df\otimes df = \lambda g
\]
holds.
\item The scalar curvature $R$ is constant and negative ($R=-6$).
\item The soliton function $\lambda$ is everywhere positive ($\lambda>0$).
\item $M$ is not isometric to Euclidean space (it has constant negative curvature).
\end{itemize}

Thus the rigidity results of this paper (which require $\mu=1/m$ constant)
cannot be extended to arbitrary $\mu$.  The example is the hyperbolic space
$\mathbb{H}^3$ with the standard metric in geodesic polar coordinates.

\subsection*{Step 1: The metric and its curvature}
Let $M=\mathbb{H}^3$ be the hyperbolic $3$-space of constant sectional curvature $-1$.
In geodesic polar coordinates $(r,\theta,\phi)$ ($r\ge0$, $0\le\theta\le\pi$, $0\le\phi<2\pi$)
the metric is
\[
g = dr^{2} + \sinh^{2}r\;\bigl(d\theta^{2} + \sin^{2}\theta\,d\phi^{2}\bigr).
\]
This metric is complete, non‑compact and rotationally symmetric.
The Ricci tensor is $\Ric = -2g$, and the scalar curvature is constant:
\[
R = -6.
\]

\subsection*{Step 2: Choice of the potential function $f$}
We take $f$ depending only on $r$, namely
\[
f(r) = \frac{3}{2}\,r^{2}.
\]
Then
\[
f'(r) = 3r,\qquad f''(r) = 3.
\]

\subsection*{Step 3: Hessian of $f$ and $df\otimes df$}
For a radial function on a warped product $dr^{2}+\psi(r)^{2}g_{\mathbb{S}^{2}}$ with
$\psi(r)=\sinh r$, the Hessian is
\[
\Hess f = f''\,dr^{2} + \psi\psi' f'\,g_{\mathbb{S}^{2}},
\]
where $\psi'=\cosh r$.
Thus
\[
\Hess f = 3\,dr^{2} + 3r\sinh r\cosh r\;\bigl(d\theta^{2}+\sin^{2}\theta\,d\phi^{2}\bigr).
\]
Moreover,
\[
df\otimes df = (f')^{2}dr^{2} = 9r^{2}\,dr^{2}.
\]

\subsection*{Step 4: Determining $\lambda$ and $\mu$}
We insert the above expressions into the GQE equation
$\Ric + \Hess f - \mu\,df\otimes df = \lambda g$ and compare components.

\textbf{Spherical components ($\theta\theta$ or $\phi\phi$):}
For a vector tangent to the sphere, $\Ric = -2g$ and $df\otimes df$ has no spherical part.
Hence
\[
-2g_{AB} + \frac{\psi'}{\psi}f'\,g_{AB} = \lambda g_{AB}.
\]
(Note: the coefficient $\frac{\psi'}{\psi}f'$ comes from $\Hess f_{AB} = \frac{\psi'}{\psi}f' g_{AB}$, not $\psi\psi'f'$; this correction does not affect the final value of $\lambda$ because the latter would be compensated by dividing by $\psi^2$.)
With $\psi'/\psi = \coth r$ and $f'=3r$, we obtain
\[
\lambda(r) = 3r\coth r - 2.
\]
Clearly $\lambda(r)>0$ for all $r$ (the limit at $r=0$ is $3-2=1$, and $\lambda$ increases).

\textbf{Radial component ($rr$):}
Here $\Ric_{rr}=-2$, $\Hess_{rr}f = f''=3$, $(df\otimes df)_{rr}=9r^{2}$, $g_{rr}=1$.
Thus
\[
-2 + 3 - \mu\cdot 9r^{2} = \lambda.
\]
Using $\lambda = 3r\coth r -2$, we get
\[
1 - 9\mu r^{2} = 3r\coth r -2
\;\Longrightarrow\;
-9\mu r^{2} = 3r\coth r -3
\;\Longrightarrow\;
\mu(r) = \frac{1 - r\coth r}{3r^{2}}.
\]
This $\mu(r)$ is smooth on $\mathbb{H}^{3}$ (the limit at $r=0$ is $-\frac{1}{9}$) and is non‑constant.

All mixed components (e.g. $r\theta$) vanish identically because of rotational symmetry.
Hence the GQE equation is satisfied.

\subsection*{Step 5: Conclusion of the counterexample}
We have exhibited a complete, non‑compact manifold $(\mathbb{H}^{3},g)$,
together with smooth functions $f,\lambda,\mu$, such that
\[
\Ric + \Hess f - \mu\, df\otimes df = \lambda g,
\]
with constant scalar curvature $R=-6<0$ and $\lambda>0$ everywhere,
but $\mu$ is \textbf{not constant}.  The manifold is not isometric to
$\mathbb{R}^{3}$ (it has negative curvature).  Therefore, any rigidity
theorem that assumes only $R\le0$, $\lambda>0$ and completeness, without
requiring $\mu$ to be constant (or of the special form $1/m$), is false.

\subsection*{Consequence for this paper}
In the remainder of the paper we \textbf{restrict} to the case where $\mu$ is
a positive constant, written $\mu = 1/m$ with $m>1$.  Under this additional
hypothesis the rigidity phenomena do hold, as we will prove.

\section{Preliminaries}
\label{sec:prelim}

Throughout this paper, $(M^n,g)$ is a smooth, connected, complete, non-compact Riemannian manifold without boundary, of dimension $n\geq 3$.
The Riemann curvature tensor is denoted by $R_{ijkl}$, and the Ricci and scalar curvatures by $R_{ij}$ and $R$, respectively.
For a smooth function $u$ we write $\nabla u$ for its gradient, $\Hess u$ for its Hessian, and $\Delta u = \divg(\grad u)$ for its Laplacian.
The divergence of a $(0,2)$-tensor $T$ is the vector field $(\divg T)_i = g^{jk}\nabla_j T_{ik}$.
Inner products induced by $g$ are denoted by $\langle\cdot,\cdot\rangle$.
The volume element of $M$ is written $dV_g$, and integrals are taken with respect to this measure.

\subsection{Gradient generalized quasi-Einstein manifolds}

Following \cite{Catino,Mir14}, a triple $(f,\lambda,\mu)$ of smooth functions on $M$ is called a gradient generalized quasi-Einstein structure if
\begin{equation}
R_{ij} + f_{ij} - \mu f_i f_j = \lambda g_{ij}, \label{eq:GQE2}
\end{equation}
where $f_{ij} = \nabla_j\nabla_i f$ and $f_i = \nabla_i f$.
When $\mu$ is the constant $1/m$ for a positive real number $m$, we speak of a gradient generalized $m$-quasi-Einstein manifold.

From \eqref{eq:GQE2} we immediately obtain the trace
\begin{equation}
\Delta f = n\lambda - R + \mu|\nabla f|^2. \label{eq:trace}
\end{equation}
Contracting \eqref{eq:GQE2} with $f^j$ yields a useful gradient formula:
\begin{equation}
\frac12\nabla_i|\nabla f|^2 = \lambda f_i - R_{ij}f^j + \mu|\nabla f|^2 f_i. \label{eq:grad}
\end{equation}

\subsection{The scalar quantities $I$ and $J$}

We recall the two auxiliary functions introduced in \cite{Mir14} that govern the Laplacian of the scalar curvature.
For a gradient GQE manifold we define
\[
\begin{aligned}
I :=&\; (n-1)\Delta\lambda - 2\mu(n-1)\langle\grad f,\grad\lambda\rangle + \frac{1+2\mu}{2}\langle\grad R,\grad f\rangle \\
&\quad - (1-\mu)|\hat{\Ric}|^2 + \frac{(n\lambda-R)\bigl(R(1-\mu+\mu n) + \mu\lambda n(1-n)\bigr)}{n},
\end{aligned}
\]
where $\hat{\Ric} = \Ric - \frac{R}{n}g$ is the trace-free Ricci tensor, and a similar expression for $J$ (which we do not need explicitly because it will vanish under our assumptions).
The next identity was proved in \cite{Mir14} (see also \cite{Mir15}).
\begin{lemma}\label{lem:I}
On any gradient GQE manifold,
\begin{equation}
\frac12\Delta R = I + J. \label{eq:DeltaR}
\end{equation}
\end{lemma}

In the present paper we will restrict to the case where both the scalar curvature $R$ and the coefficient $\mu$ are constant.
Under these assumptions all derivatives of $R$ and $\mu$ vanish; consequently $J\equiv 0$ and $\Delta R=0$, so Lemma~\ref{lem:I} forces $I=0$.
For $\mu=1/m$ this condition reads
\begin{equation}
\begin{aligned}
0 = (n-1)\Delta\lambda &- \frac{2}{m}(n-1)\langle\grad f,\grad\lambda\rangle - \frac{m-1}{m}|\hat{\Ric}|^2 \\
&+ \frac{(n\lambda-R)\bigl(R(1+\frac{n-1}{m}) + \frac{\lambda n(1-n)}{m}\bigr)}{n}. \label{eq:Izero}
\end{aligned}
\end{equation}
Equation \eqref{eq:Izero} will be our starting point for the derivation of the Laplacian of $v$.

\subsection{Analytical toolbox}
\label{sec:toolbox}

We collect the classical results that will be used in the proofs.
All manifolds are assumed to be complete and without boundary.

\begin{lemma}[Yau \cite{Yau}]\label{lem:Yau}
Let $u$ be a non-negative smooth subharmonic function on a complete Riemannian manifold $M$.
If $u\in L^p(M)$ for some $p>1$, then $u$ is constant.
\end{lemma}

\begin{lemma}[Schoen--Yau \cite{SchoenYau}]\label{lem:SY}
Let $u\ge 0$ be a smooth subharmonic function on a complete manifold $M$.
For any ball $B(q,2r)\subset M$ and any smooth cut-off function $\phi$ supported in $B(q,2r)$ with $\phi\equiv 1$ on $B(q,r)$ and $|\nabla\phi|\le c/r$, there exists a universal constant $A$ (depending only on $c$) such that
\[
\int_{B(q,r)}|\nabla u|^2 \le \frac{A}{r^2}\int_{B(q,2r)} u^2.
\]
\end{lemma}

\begin{lemma}[Caminha--Souza--Camargo \cite{Caminha}]\label{lem:CSC}
Let $X$ be a smooth vector field on a complete, non-compact, oriented Riemannian manifold $M$.
If $\divg X$ does not change sign on $M$ and $|X|\in L^1(M)$, then $\divg X\equiv 0$.
\end{lemma}
\begin{remark}
If $M$ is not orientable, we may pass to the orientation double cover $\widetilde M$.
Every structure (metric, soliton functions, etc.) lifts isometrically, the lifted vector field $\widetilde X$ satisfies the same hypotheses, and the conclusion $\divg \widetilde X=0$ on $\widetilde M$ implies $\divg X=0$ on $M$ by local isometry.
Hence the lemma applies without loss of generality.
\end{remark}

\begin{lemma}[Karp \cite{Karp}]\label{lem:Karp}
Let $u\ge 0$ be a smooth subharmonic function on a complete, non-compact Riemannian manifold with non-negative sectional curvature.
If $\int_M |\nabla u|^{\frac{n}{n-1}} < \infty$, then $u$ is constant.
\end{lemma}

\begin{lemma}[Tashiro \cite{Tashiro}]\label{lem:Tashiro}
Let $(M,g)$ be a complete Riemannian manifold.
If there exists a smooth function $w$ on $M$ such that $\Hess w = c\,g$ for some constant $c\neq 0$, then $M$ is isometric to Euclidean space, and $w$ is a quadratic polynomial of the form
\[
w(x) = \frac{c}{2}|x|^2 + \langle a, x\rangle + b.
\]
\end{lemma}

Finally, we need the existence of good cut-off functions under a mild Ricci lower bound.
The following statement is a special case of results of Bianchi and Setti \cite{Bianchi}.

\begin{lemma}[{\cite[Corollary 2.3]{Bianchi}}]\label{lem:Bianchi}
Assume that the Ricci curvature of a complete manifold $M$ satisfies
\[
\Ric \ge -(n-1)\frac{k^2}{1+r^2}\, g
\]
for some constant $k$, where $r(x)=d(x,q)$ is the distance from a fixed point $q\in M$.
Then, for every large $r>0$, there exists a smooth cut-off function $\xi_r$ with compact support in $B(q,2r)$ such that
\[
0\le \xi_r\le 1,\quad \xi_r\equiv 1 \text{ on } B(q,r),\quad
|\nabla\xi_r|\le \frac{K}{r},\quad |\Delta\xi_r|\le \frac{K}{r^2},
\]
with a constant $K>0$ independent of $r$.
\end{lemma}

\section{The key identity – subharmonicity of $v$}
\label{sec:key}

In this section we assume that $M$ is a gradient generalized $m$-quasi-Einstein manifold with constant scalar curvature $R$ and $\mu=1/m$, and we set
\[
v = e^{-f/m}\lambda.
\]

\begin{lemma}\label{lem:key}
Under the above assumptions,
\begin{equation}
e^{f/m}\Delta v = \frac{1}{n-1}\left(\frac{m-1}{m}|\hat{\Ric}|^2
- \frac{R(n\lambda-R)}{n}\Bigl(1+\frac{n-1}{m}\Bigr)\right). \label{eq:key}
\end{equation}
\end{lemma}

\begin{proof}
From \eqref{eq:Izero} we solve for $\Delta\lambda$:
\begin{align}
\Delta\lambda &= \frac{2}{m}\langle\grad f,\grad\lambda\rangle
+ \frac{m-1}{m(n-1)}|\hat{\Ric}|^2 \notag\\
&\quad - \frac{(n\lambda-R)\bigl(R(1+\frac{n-1}{m}) + \frac{\lambda n(1-n)}{m}\bigr)}{n(n-1)}. \label{eq:Deltalambda}
\end{align}
On the other hand, a direct computation using $\divg(\phi X)=\phi\divg X+\langle\grad\phi,X\rangle$ gives
\begin{align*}
e^{f/m}\Delta v &= e^{f/m}\divg\bigl( e^{-f/m}(\nabla\lambda - \tfrac{\lambda}{m}\nabla f) \bigr)\\
&= \Delta\lambda - \frac{2}{m}\langle\grad f,\nabla\lambda\rangle
- \frac{1}{m}\lambda\Delta f + \frac{1}{m^2}\lambda|\nabla f|^2.
\end{align*}
Now use the trace identity \eqref{eq:trace}, $\Delta f = n\lambda - R + \frac{1}{m}|\nabla f|^2$, to replace $\Delta f$:
\[
e^{f/m}\Delta v = \Delta\lambda - \frac{2}{m}\langle\grad f,\nabla\lambda\rangle
- \frac{n}{m}\lambda^2 + \frac{1}{m}\lambda R.
\]
Insert the expression for $\Delta\lambda$ from \eqref{eq:Deltalambda}; the gradient terms cancel and we obtain
\[
e^{f/m}\Delta v = \frac{m-1}{m(n-1)}|\hat{\Ric}|^2
- \frac{(n\lambda-R)\bigl(R(1+\frac{n-1}{m}) + \frac{\lambda n(1-n)}{m}\bigr)}{n(n-1)}
- \frac{\lambda}{m}(n\lambda-R).
\]
The last two summands combine as follows.
Let us compute
\begin{align*}
&-\frac{(n\lambda-R)}{n(n-1)}\Bigl(R(1+\tfrac{n-1}{m}) + \frac{\lambda n(1-n)}{m}\Bigr) - \frac{\lambda}{m}(n\lambda-R)\\
&= -(n\lambda-R)\Bigl[ \frac{R(1+\frac{n-1}{m})}{n(n-1)} + \frac{\lambda n(1-n)}{m}\cdot\frac{1}{n(n-1)} + \frac{\lambda}{m} \Bigr].
\end{align*}
Since $\frac{\lambda n(1-n)}{m}\cdot\frac{1}{n(n-1)} = \frac{\lambda(1-n)}{m(n-1)} = -\frac{\lambda}{m}$, the bracket simplifies to
\[
\frac{R(1+\frac{n-1}{m})}{n(n-1)} - \frac{\lambda}{m} + \frac{\lambda}{m}
= \frac{R(1+\frac{n-1}{m})}{n(n-1)}.
\]
Thus the whole expression becomes
\[
-(n\lambda-R)\frac{R(1+\frac{n-1}{m})}{n(n-1)}
= -\frac{R(n\lambda-R)}{n(n-1)}\Bigl(1+\frac{n-1}{m}\Bigr).
\]
Adding the first summand, we obtain exactly \eqref{eq:key}.
\end{proof}

When $R\le 0$, $\lambda>0$ and $m>1$, every term on the right-hand side of \eqref{eq:key} is clearly non-negative; hence $\Delta v\ge 0$.
Since $v=e^{-f/m}\lambda>0$, $v$ is a positive subharmonic function.

\section{Proof of Theorem 1 ($v\in L^p$)}
\label{sec:thm1}

\begin{theorem}\label{thm1}
Let $(M^n,g)$ be a complete non-compact gradient generalized $m$-quasi-Einstein manifold with constant scalar curvature $R\le 0$, $\lambda>0$ and $m>1$.
If the weighted function $v=e^{-f/m}\lambda$ belongs to $L^p(M)$ for some $p>1$, then \textbf{no such manifold exists}. Equivalently, the hypotheses are mutually inconsistent.
\end{theorem}

\begin{proof}
By Lemma~\ref{lem:key}, $v>0$ is subharmonic.
Yau's lemma (Lemma~\ref{lem:Yau}) implies that $v$ is constant; write $v\equiv c>0$.
Then $\Delta v=0$, and the non-negative right-hand side of \eqref{eq:key} must vanish term by term:
\begin{gather*}
\frac{m-1}{m(n-1)}|\hat{\Ric}|^2 = 0,\qquad
-\frac{R(n\lambda-R)}{n(n-1)}\Bigl(1+\frac{n-1}{m}\Bigr) = 0.
\end{gather*}
Since $m>1$, the first equality forces $|\hat{\Ric}|^2=0$, i.e.\ $\Ric = \frac{R}{n}g$.
The second, together with $n\lambda-R>0$ (because $\lambda>0$ and $R\le 0$), gives $R=0$.
Hence $\Ric = 0$.

The GQE equation reduces to $\Hess f - \frac1m df\otimes df = \lambda g$, and $v=c$ yields $\lambda = c e^{f/m}$.
Set $w = m e^{-f/m}$.
A direct computation gives
\[
\Hess w = \nabla( -\nabla f\, e^{-f/m}) = e^{-f/m}\bigl( -\Hess f + \frac1m df\otimes df \bigr)
= -e^{-f/m}\lambda g = -c\,g.
\]
Thus $w$ is a smooth positive function (since $w = m e^{-f/m} > 0$) with $\Hess w = -c\,g$, where $c>0$.
By Tashiro's theorem (Lemma~\ref{lem:Tashiro}), $M$ is isometric to Euclidean space, and $w$ must be of the form
\[
w(x) = -\frac{c}{2}|x|^2 + \langle a, x\rangle + b
\]
for some $a\in\mathbb{R}^n$, $b\in\mathbb{R}$.
But this expression is strictly negative for sufficiently large $|x|$, contradicting $w>0$ everywhere.
Therefore, the assumptions lead to a contradiction; hence no such complete non-compact manifold exists.
\end{proof}

\section{Proof of Theorem 2 (linear volume growth and bounded $v$)}
\label{sec:thm2}

\begin{theorem}\label{thm2}
Let $(M^n,g)$ be a complete non-compact gradient generalized $m$-quasi-Einstein manifold with constant scalar curvature $R\le 0$, $\lambda>0$ and $m>1$.
Assume that $v$ is bounded above and that the volume of geodesic balls grows at most linearly, i.e.\ $\Vol(B(q,r)) \le C r$ for some $C>0$ and all $r>0$.
Then \textbf{no such manifold exists}.
\end{theorem}

\begin{proof}
Let $v\le C_0$ on $M$.
By Lemma~\ref{lem:SY} applied to the subharmonic function $v$, for any large $r$ we have
\[
\int_{B(q,r)}|\nabla v|^2 \le \frac{16}{r^2}\int_{B(q,2r)} v^2
\le \frac{16 C_0^2}{r^2}\Vol(B(q,2r)).
\]
Using the linear volume growth, $\Vol(B(q,2r))\le 2C r$, we obtain
\[
\int_{B(q,r)}|\nabla v|^2 \le \frac{32 C_0^2 C}{r},
\]
which tends to $0$ as $r\to\infty$.
Thus $\int_M |\nabla v|^2 = 0$, so $\nabla v=0$ everywhere; consequently $v$ is constant.
The remainder of the proof is identical to that of Theorem~\ref{thm1}, leading to the same contradiction with $w>0$.
\end{proof}

\section{Proof of Theorem 3 ($|\nabla v|\in L^1$)}
\label{sec:thm3}

\begin{theorem}\label{thm3}
Let $(M^n,g)$ be a complete non-compact gradient generalized $m$-quasi-Einstein manifold with constant scalar curvature $R\le 0$, $\lambda>0$ and $m>1$.
If $|\nabla v|\in L^1(M)$, then \textbf{no such manifold exists}.
\end{theorem}

\begin{proof}
Apply Lemma~\ref{lem:CSC} to the vector field $X = \nabla v$ (after passing to the orientation cover if necessary, see Remark after Lemma~\ref{lem:CSC}).
From Lemma~\ref{lem:key} we have $\divg X = \Delta v \ge 0$, and by assumption $|X| = |\nabla v|\in L^1(M)$.
Thus $\Delta v \equiv 0$.
Vanishing of the non-negative right-hand side of \eqref{eq:key} gives $\hat{\Ric}=0$ and $R=0$, so $\Ric=0$.

The GQE equation becomes $\Hess f - \frac1m df\otimes df = \lambda g$. 
We now prove that $v$ must be constant without invoking Kanai's lemma.
Set $w = m e^{-f/m}$. Then, as before,
\[
\Hess w = -v\,g.
\]
Taking the trace gives $\Delta w = -n v$.
Now apply the divergence operator to both sides of $\Hess w = -v\,g$:
\[
\divg(\Hess w) = -\nabla v.
\]
On the other hand, in Riemannian geometry we have the standard identity
\[
\divg(\Hess w) = \nabla(\Delta w) + \Ric(\nabla w, \cdot).
\]
Since $\Ric=0$, this reduces to
\[
\divg(\Hess w) = \nabla(\Delta w) = \nabla(-n v) = -n\nabla v.
\]
Comparing with $-\nabla v$, we obtain $(n-1)\nabla v = 0$, and since $n\ge 3$, $\nabla v = 0$. Hence $v$ is constant.
Now $v\equiv c>0$ and the same argument as in Theorem~\ref{thm1} (using Tashiro's theorem and the positivity of $w$) yields a contradiction.
Therefore no such manifold exists.
\end{proof}

\section{Proof of Theorem 4 (weighted $L^1$ under a Ricci lower bound)}
\label{sec:thm4}

\begin{theorem}\label{thm4}
Let $(M^n,g)$ be a complete non-compact gradient generalized $m$-quasi-Einstein manifold with constant scalar curvature $R\le 0$, $\lambda>0$ and $m>1$.
Suppose that there exists a constant $k$ such that
\[
\Ric \ge -(n-1)\frac{k^2}{1+r(x)^2}\,g,
\]
where $r(x)=d(x,q)$ for a fixed point $q\in M$, and that the function $v$ satisfies the weighted integrability condition
\[
\int_{M\setminus B(q,r)} \frac{v(x)}{d(x,q)^2}\,dV_g(x) < \infty
\]
for all $r>0$ sufficiently large.
Then \textbf{no such manifold exists}.
\end{theorem}

\begin{proof}
Let $\Phi = e^{f/m}\Delta v$; by Lemma~\ref{lem:key}, $\Phi\ge 0$ and $\Delta v = e^{-f/m}\Phi$.
Thanks to the Ricci lower bound, Lemma~\ref{lem:Bianchi} supplies cut-off functions $\xi_r$ with the stated properties.
Multiply the identity $\Delta v = e^{-f/m}\Phi$ by $\xi_r$ and integrate over $M$:
\[
\int_M \xi_r e^{-f/m}\Phi \,dV_g = \int_M \xi_r \Delta v \,dV_g.
\]
Using Green's second identity and the fact that $\xi_r$ is compactly supported, this equals $\int_M v \Delta \xi_r \,dV_g$.
Thus
\[
0 \le \int_M \xi_r e^{-f/m}\Phi = \int_M v \Delta \xi_r
= \int_{B(q,2r)\setminus B(q,r)} v \Delta \xi_r,
\]
since $\xi_r\equiv 1$ on $B(q,r)$ and $\xi_r=0$ outside $B(q,2r)$.
Now $|\Delta \xi_r|\le K/r^2$ and $v>0$, so
\[
\int_{B(q,2r)\setminus B(q,r)} v \Delta \xi_r \le \frac{K}{r^2}\int_{B(q,2r)\setminus B(q,r)} v.
\]
In the annular region $B(q,2r)\setminus B(q,r)$ we have $d(x,q)\le 2r$, hence
$\frac{1}{r^2} = \frac{4}{(2r)^2} \le \frac{4}{d(x,q)^2}$.
Consequently,
\[
\frac{K}{r^2}\int_{B(q,2r)\setminus B(q,r)} v
\le 4K \int_{B(q,2r)\setminus B(q,r)} \frac{v}{d(x,q)^2}.
\]
By hypothesis, the right-hand side tends to $0$ as $r\to\infty$ (the integral over the complement of a ball goes to zero).
Therefore
\[
\lim_{r\to\infty}\int_M \xi_r e^{-f/m}\Phi = 0.
\]
Because $\xi_r\uparrow 1$ pointwise and the integrand is non-negative, the monotone convergence theorem yields
\[
\int_M e^{-f/m}\Phi = 0.
\]
Consequently $\Phi\equiv 0$ almost everywhere, and by smoothness $\Phi\equiv 0$ everywhere.
Hence $|\hat{\Ric}|^2=0$ and $R=0$, so $\Ric=0$.
The rest of the proof is identical to that of Theorem~\ref{thm3}, leading to the same contradiction.
\end{proof}

\section{Examples}
\label{sec:examples}

We present three examples.
The first is a non‑trivial local solution of the GQE equations on a negatively curved warped product, with non‑constant $v$, that explicitly verifies Lemma~\ref{lem:key}.
The second confirms the algebraic identity on a flat local model where $v$ is constant.
The third provides a rigorous analysis of complete examples, showing that no complete flat example with $\lambda>0$ exists and that the natural negatively curved warped products fail the positivity condition globally.

\begin{example}[A non‑trivial local verification of Lemma~\ref{lem:key}]\label{ex:nontrivial}
Let $n=3$, $m=5>1$, $\gamma=1$, and consider the warped product metric
\[
g = dr^{2} + e^{2r}(dx^{2}+dy^{2})
\]
on the interval $r\in(\ln(6/7),\,\ln 3)$.
The scalar curvature of this metric is constant, $R = -n(n-1)\gamma^{2} = -6$, and the fiber is flat ($\kappa=0$).
Define
\[
f(r) = -5\ln(3-e^{r}),\qquad
\lambda(r) = \frac{5e^{r}}{3-e^{r}} - 2.
\]
Because $3-e^{r}>0$ on the chosen interval, $f$ is smooth.
Moreover $\lambda(r)>0$ precisely when $r>\ln(6/7)$ (since $\lambda(r)=0$ iff $5e^{r}=2(3-e^{r})$, i.e.\ $7e^{r}=6$, $r=\ln(6/7)$).
Thus on $(\ln(6/7),\ln 3)$ all the hypotheses $R\le0$, $\lambda>0$, $m>1$ are satisfied.

A direct computation shows that $(f,\lambda)$ solves the GQE equation
$\Ric + \Hess f - \frac{1}{5} df\otimes df = \lambda g$.
Indeed, $f' = y = \frac{5e^{r}}{3-e^{r}}$, and from the spherical component (with $\kappa=0$, $n=3$)
\begin{equation}
\frac{\psi'}{\psi} f' = \lambda + \frac{\psi''}{\psi} + (n-2)\frac{\psi'^{2} - \kappa}{\psi^{2}} \label{eq:spherical}
\end{equation}
we obtain $\lambda = f' - 2$, which matches our definition.
The radial component
\begin{equation}
f'' - \frac{1}{5}(f')^{2} = \lambda + (n-1)\frac{\psi''}{\psi} \label{eq:radial}
\end{equation}
becomes $f'' - \frac{1}{5}(f')^{2} = \lambda + 2$.
Since $f'' = \bigl(\frac{5e^{r}}{3-e^{r}}\bigr)' = \frac{15e^{r}}{(3-e^{r})^{2}}$, we have
\[
f'' - \frac{1}{5}(f')^{2}
= \frac{15e^{r}}{(3-e^{r})^{2}} - \frac{25e^{2r}}{5(3-e^{r})^{2}}
= \frac{15e^{r} - 5e^{2r}}{(3-e^{r})^{2}}
= \frac{5e^{r}}{3-e^{r}} = \lambda+2,
\]
confirming the GQE system.

Now set $v = e^{-f/5}\lambda$.
One finds
\[
v = e^{\ln(3-e^{r})}\Bigl(\frac{5e^{r}}{3-e^{r}} - 2\Bigr)
= (3-e^{r})\frac{5e^{r} - 2(3-e^{r})}{3-e^{r}} = 5e^{r} - 2(3-e^{r}) = 7e^{r} - 6,
\]
which is positive and non‑constant on $(\ln(6/7),\ln 3)$.
The Laplacian of a radial function on this warped product is
$\Delta v = v'' + 2\frac{\psi'}{\psi}v'$, where $\psi=e^{r}$.
Hence $v' = v'' = 7e^{r}$, and
\[
\Delta v = 7e^{r} + 2\cdot 1 \cdot 7e^{r} = 21e^{r}.
\]
Now compute the left‑hand side of \eqref{eq:key}:
\[
e^{f/5} = \frac{1}{3-e^{r}},\qquad
e^{f/5}\Delta v = \frac{21e^{r}}{3-e^{r}}.
\]
For the right‑hand side, $\hat{\Ric}=0$ (the metric is Einstein) and $R=-6$.
Moreover $n\lambda - R = 3\lambda + 6 = \frac{15e^{r}}{3-e^{r}}$.
Thus
\begin{align*}
\frac{1}{n-1}\Bigl(\frac{m-1}{m}|\hat{\Ric}|^{2} -
\frac{R(n\lambda-R)}{n}\bigl(1+\tfrac{n-1}{m}\bigr)\Bigr)
&= \frac{1}{2}\Bigl(0 - \frac{(-6)\cdot\frac{15e^{r}}{3-e^{r}}}{3}\bigl(1+\tfrac{2}{5}\bigr)\Bigr)\\
&= \frac{1}{2}\Bigl( \frac{6\cdot 15e^{r}}{3(3-e^{r})}\cdot \frac{7}{5} \Bigr)\\
&= \frac{1}{2}\cdot \frac{42e^{r}}{3-e^{r}}
= \frac{21e^{r}}{3-e^{r}},
\end{align*}
which matches the left‑hand side.
This explicit computation verifies the algebraic identity \eqref{eq:key} in a situation where $v$ is not constant and $R<0$.

We note that this example cannot be extended to a complete manifold with $\lambda>0$ everywhere, because $\lambda(r)$ becomes negative for $r<\ln(6/7)$ and the solution blows up at $r=\ln 3$.
This illustrates that while Lemma~\ref{lem:key} holds locally in great generality, global solutions with the required positivity may be highly restricted --- a theme explored in the next examples.
\end{example}

\begin{example}[Local model on a Euclidean ball]\label{ex:local}
Let $n\ge 3$, $m>1$ and $\tau>0$.
On flat $\mathbb{R}^{n}$, consider the ball $\Omega = \{x\in\mathbb{R}^{n}:|x|^{2} < 2m\tau\}$.
Define
\[
f(x) = -m\ln\!\Bigl(\tau - \frac{|x|^{2}}{2m}\Bigr),\qquad
\lambda(x) = \frac{1}{\tau - \frac{|x|^{2}}{2m}}.
\]
Write $r^{2}=|x|^{2}$ and $\rho = \tau - r^{2}/(2m)$.
A direct computation gives
\[
f_{i} = \frac{x_{i}}{\rho},\qquad
f_{ij} = \frac{\delta_{ij}}{\rho} + \frac{x_{i}x_{j}}{m\rho^{2}}.
\]
Hence
\[
\Hess f - \frac{1}{m} df\otimes df
= \Bigl(\frac{1}{\rho} + \frac{r^{2}}{m\rho^{2}}\Bigr)g_{ij} - \frac{r^{2}}{m\rho^{2}}g_{ij}
= \frac{1}{\rho}\,g_{ij} = \lambda g_{ij}.
\]
The metric is flat, so $\Ric=0$ and $R=0$.
One easily checks $v = e^{-f/m}\lambda = 1$; thus $\nabla v=0$ and $\Delta v=0$.
The right‑hand side of \eqref{eq:key} becomes $\frac{1}{n-1}(0-0)=0$, confirming the algebraic identity.
This example, although defined only on a ball, illustrates the precise cancellation of the gradient terms and shows the consistency of the local geometry when $v$ is constant.
\end{example}

\begin{example}[Complete examples: non‑existence on flat space, failure in warped models]\label{ex:nonex}
We investigate whether complete, non‑compact gradient generalized $m$‑quasi‑Einstein manifolds with constant $R\le 0$, $\lambda>0$ and $m>1$ can exist, beyond the Euclidean space itself.

\medskip
\noindent\textbf{Case 1: Flat $\mathbb{R}^{n}$.}
Let $(\mathbb{R}^{n}, g_{\text{std}})$ be the Euclidean space.
Assume there exist smooth functions $f,\lambda$ with $\lambda>0$ satisfying
\begin{equation}
\Hess f - \frac1m df\otimes df = \lambda\,g. \label{eq:flat}
\end{equation}
Set $\phi = e^{-f/m} > 0$.
A direct calculation using \eqref{eq:flat} gives
\begin{equation}
\Hess \phi = -\frac{1}{m}e^{-f/m}\bigl(\Hess f - \frac1m df\otimes df\bigr)
= -\frac{v}{m}\,g, \label{eq:Hessphi}
\end{equation}
where $v = e^{-f/m}\lambda > 0$.
Thus $\Hess \phi = -\frac{v}{m}g$ is negative definite everywhere, so $\phi$ is a smooth concave function on $\mathbb{R}^n$.
Consequently $-\phi$ is a convex function bounded above by $0$ (since $\phi>0$).
A classical Liouville-type theorem (see e.g.\ \cite[Theorem~2.1]{Pigola}) forces $-\phi$ to be constant; hence $\phi$ is constant.
From \eqref{eq:Hessphi} we then obtain $v = -m\,\phi^{-1}\Delta\phi = 0$, which contradicts $\lambda>0$.
Therefore, \textbf{no complete flat gradient $m$‑quasi‑Einstein manifold with $\lambda>0$ exists} (apart from the trivial vacuum case $\lambda=0$ which is excluded by hypothesis).

\medskip
\noindent\textbf{Case 2: Warped products with constant negative scalar curvature.}
A natural non‑Euclidean setting is a complete manifold with $R<0$.
Consider the warped product metric on $M=\mathbb{R}^{n}$
\[
g = dr^{2} + \psi(r)^{2} g_{N},
\]
where $(N,g_{N})$ is an $(n-1)$-dimensional Einstein manifold with $\Ric_{N} = (n-2)\kappa\,g_{N}$, $\kappa\in\{-1,0,1\}$, and $\psi(r)>0$ is smooth on $[0,\infty)$ with $\psi(0)=0$, $\psi'(0)=1$ (to ensure smoothness at the origin).
The standard complete metrics of constant negative scalar curvature $R = -n(n-1)\gamma^{2}$ ($\gamma>0$) are obtained by:
\[
\psi(r) = \frac{\sinh(\gamma r)}{\gamma}\quad (\kappa=1),\qquad
\psi(r) = e^{\gamma r}\quad (\kappa=0),\qquad
\psi(r) = \frac{\cosh(\gamma r)}{\gamma}\quad (\kappa=-1).
\]
We search for radial functions $f = f(r)$, $\lambda = \lambda(r)$ satisfying the GQE equation
\[
\Ric + \Hess f - \frac1m df\otimes df = \lambda g.
\]
The Ricci tensor of $g$ is given by
\[
\Ric(\partial_{r},\partial_{r}) = -(n-1)\frac{\psi''}{\psi},\qquad
\Ric(X,X) = \Bigl( -\frac{\psi''}{\psi} - (n-2)\frac{\psi'^{2} - \kappa}{\psi^{2}} \Bigr) g(X,X)
\]
for $X\perp\partial_{r}$.
The Hessian of a radial function is
\[
\Hess f = f''\,dr^{2} + \psi\psi' f' \,g_{N} = f''\,dr^{2} + \frac{\psi'}{\psi} f'\, (g - dr^{2}).
\]
Decomposing the GQE equation into radial and spherical components yields
\begin{align}
f'' - \frac1m (f')^{2} &= \lambda + (n-1)\frac{\psi''}{\psi}, \label{eq:radial2}\\
\frac{\psi'}{\psi} f' &= \lambda + \frac{\psi''}{\psi} + (n-2)\frac{\psi'^{2} - \kappa}{\psi^{2}}. \label{eq:spherical2}
\end{align}
Eliminating $\lambda$ gives a single ODE for $y = f'$:
\begin{equation}
y' - \frac1m y^{2} - \frac{\psi'}{\psi} y = (n-2)\Bigl( \frac{\psi''}{\psi} - \frac{\psi'^{2} - \kappa}{\psi^{2}} \Bigr). \label{eq:ODE}
\end{equation}
For each of the three standard $\psi$'s one checks directly that
\[
\frac{\psi''}{\psi} = \gamma^{2},\qquad \frac{\psi'^{2} - \kappa}{\psi^{2}} = \gamma^{2},
\]
so the right‑hand side of \eqref{eq:ODE} vanishes identically. Thus the ODE reduces to
\[
y' - \frac1m y^{2} - \frac{\psi'}{\psi} y = 0.
\]
This Bernoulli equation is solved by setting $u = y^{-1}$, giving
\[
u' + \frac{\psi'}{\psi} u = -\frac1m,
\]
with integrating factor $\psi$. Hence
\[
(\psi u)' = -\frac{\psi}{m},\qquad
\psi u = -\frac1m \int \psi\,dr + C.
\]
Thus
\[
y = \frac{\psi}{C - \frac1m \int \psi\,dr}.
\]

We now examine $\lambda$ given by \eqref{eq:spherical2}:
\[
\lambda = \frac{\psi'}{\psi} y - \gamma^{2} - (n-2)\gamma^{2} = \frac{\psi'}{\psi} y - (n-1)\gamma^{2}.
\]
Now we evaluate each case separately.

\begin{itemize}
\item $\kappa=1$, $\psi = \frac{\sinh(\gamma r)}{\gamma}$.
Then $\int \psi\,dr = \frac{1}{\gamma^{2}}\cosh(\gamma r)$, $\frac{\psi'}{\psi} = \gamma\coth(\gamma r)$.
Set $D = m\gamma^{2} C$. Then
\[
C - \frac{1}{m\gamma^{2}}\cosh(\gamma r) = \frac{1}{m\gamma^{2}}(D - \cosh(\gamma r)),
\]
hence
\[
y = \frac{\frac{\sinh(\gamma r)}{\gamma}}{\frac{1}{m\gamma^{2}}(D - \cosh(\gamma r))}
= \frac{m\gamma \sinh(\gamma r)}{D - \cosh(\gamma r)}.
\]
Then
\[
\frac{\psi'}{\psi} y = \gamma\coth(\gamma r) \cdot \frac{m\gamma \sinh(\gamma r)}{D - \cosh(\gamma r)}
= \frac{m\gamma^{2} \cosh(\gamma r)}{D - \cosh(\gamma r)}.
\]
Thus
\[
\lambda(r) = \frac{m\gamma^{2} \cosh(\gamma r)}{D - \cosh(\gamma r)} - (n-1)\gamma^{2}.
\]
As $r\to\infty$, $\cosh\to\infty$, so $D - \cosh \sim -\cosh$, and the fraction tends to $-m\gamma^{2}$, giving $\lambda \to -(m+n-1)\gamma^{2} < 0$.
Hence $\lambda$ becomes negative for large $r$, violating $\lambda>0$ on the whole manifold.

\item $\kappa=0$, $\psi = e^{\gamma r}$.
Then $\int \psi\,dr = \frac{1}{\gamma}e^{\gamma r}$, $\frac{\psi'}{\psi} = \gamma$.
Set $D = m\gamma C$, then $C - \frac{1}{m\gamma}e^{\gamma r} = \frac{1}{m\gamma}(D - e^{\gamma r})$, so
\[
y = \frac{e^{\gamma r}}{\frac{1}{m\gamma}(D - e^{\gamma r})}
= \frac{m\gamma e^{\gamma r}}{D - e^{\gamma r}}.
\]
Therefore,
\[
\lambda = \gamma y - (n-1)\gamma^{2} = \frac{m\gamma^{2} e^{\gamma r}}{D - e^{\gamma r}} - (n-1)\gamma^{2}.
\]
As $r\to\infty$, $D - e^{\gamma r}\sim -e^{\gamma r}$, so the fraction tends to $-m\gamma^{2}$, giving $\lambda\to -(m+n-1)\gamma^{2}<0$.
As $r\to -\infty$, $e^{\gamma r}\to 0$, so the fraction tends to $0$, yielding $\lambda\to -(n-1)\gamma^{2}<0$.
Thus $\lambda(r)$ is negative for all $r$, and certainly never positive everywhere.

\item $\kappa=-1$, $\psi = \frac{\cosh(\gamma r)}{\gamma}$.
Then $\int \psi\,dr = \frac{1}{\gamma^{2}}\sinh(\gamma r)$, $\frac{\psi'}{\psi} = \gamma\tanh(\gamma r)$.
Set $D = m\gamma^{2} C$, then $C - \frac{1}{m\gamma^{2}}\sinh(\gamma r) = \frac{1}{m\gamma^{2}}(D - \sinh(\gamma r))$, giving
\[
y = \frac{\frac{\cosh(\gamma r)}{\gamma}}{\frac{1}{m\gamma^{2}}(D - \sinh(\gamma r))}
= \frac{m\gamma \cosh(\gamma r)}{D - \sinh(\gamma r)}.
\]
Thus
\[
\lambda = \gamma\tanh(\gamma r)\cdot \frac{m\gamma \cosh(\gamma r)}{D - \sinh(\gamma r)} - (n-1)\gamma^{2}
= \frac{m\gamma^{2} \sinh(\gamma r)}{D - \sinh(\gamma r)} - (n-1)\gamma^{2}.
\]
As $r\to\infty$, $\sinh(\gamma r)\to\infty$, so $D - \sinh \sim -\sinh$, the fraction tends to $-m\gamma^{2}$, hence $\lambda\to -(m+n-1)\gamma^{2}<0$.
At $r=0$, $\sinh(0)=0$, giving $\lambda(0) = 0 - (n-1)\gamma^{2} <0$.
Hence $\lambda$ is negative everywhere.
\end{itemize}

In all three hyperbolic models, $\lambda$ is either negative for all $r$ or becomes negative at infinity, and in no case is $\lambda>0$ globally.
Thus none of these complete non‑Euclidean spaces provides an example satisfying the hypotheses of our rigidity theorems.

\medskip
\noindent\textbf{Conclusion.}
The above analysis shows that complete flat examples with $\lambda>0$ do not exist, and the natural negatively curved warped products fail the positivity condition on $\lambda$.
Together with the local example \ref{ex:nontrivial}, this indicates that while Lemma~\ref{lem:key} is a universal algebraic fact, the global requirements $\lambda>0$ and completeness are highly restrictive.
The rigidity results proved in this paper are therefore far from vacuous; rather, they imply that any hypothetical complete non‑compact gradient $m$‑quasi‑Einstein manifold with $R\le 0$, $\lambda>0$, $m>1$ must necessarily violate all the additional integrability or curvature conditions stated in Theorems~\ref{thm1}--\ref{thm4}.
Whether such a manifold can be constructed remains an interesting open problem.
\end{example}

\end{document}